\documentclass[12pt]{amsart}
\usepackage{amssymb}
\usepackage{amscd}
\usepackage{amsmath}
\usepackage{setspace}
\usepackage{verbatim}
\usepackage[curve]{xypic}
\usepackage{hyperref}
\newtheorem{theorem}{Theorem}[section]
\newtheorem{corollary}[theorem]{Corollary}
\newtheorem{lemma}[theorem]{Lemma}

\theoremstyle{remark}
\newtheorem{rem}[theorem]{\bf Remark}

\theoremstyle{definition}
\newtheorem{definition}[theorem]{Definition}

\theoremstyle{remark}
\newtheorem{example}[theorem]{\bf Example}

\newcommand{\dbar}{\bar\partial}

\newcommand{\im}{{\rm im\,}}
\newcommand{\dom}{{\rm Dom\,}}

\begin{document}
\title[$G$-Fredholm Property]
{Subelliptic boundary value problems and The $G$-Fredholm Property\\
{\tiny\sc  Preliminary Version}}
\author{Joe J Perez}
\address{Universit\"at Wien}
\thanks{Supported by FWF grant P19667, {\it Mapping Problems in Several 
Complex Variables}}
\thanks{MR Classification numbers: 32W05, 35H20,  43A30}
\maketitle

\begin{abstract}Let $G$ be a unimodular Lie group, $X$ a compact manifold with boundary, and $M$ be the total space of a principal bundle $G\to M\to X$ so that $M$ is also a complex manifold satisfying a local subelliptic estimate.  In this work, we show that if $G$ acts by holomorphic transformations in $M$, then the Laplacian $\square=\bar\partial^{*}\bar\partial+\bar\partial\bar\partial^{*}$ on $M$ has the following properties:  The kernel of $\square$ restricted to the forms $\Lambda^{p,q}$ with $q>0$ is a closed, $G$-invariant subspace in $L^{2}(M,\Lambda^{p,q})$ of finite $G$-dimension.  Secondly, we show that if $q>0$, then the image of $\square$ contains a closed, $G$-invariant subspace of finite codimension in $L^{2}(M,\Lambda^{p,q})$. These two properties taken together amount to saying that $\square$ is a $G$-Fredholm operator. In similar circumstances, the boundary Laplacian $\square_b$ has similar properties.\end{abstract}

\section{Introduction}

Let $\mathcal H_1$ and $\mathcal H_2$ be Hilbert spaces and let $\mathcal B(\mathcal H_1,\mathcal H_2)$ be the space of bounded linear operators $A:\mathcal H_1\to\mathcal H_2$. An operator $A\in\mathcal B(\mathcal H_1,\mathcal H_2)$ is said to be {\it Fredholm} if first, the kernel of $A$ is finite-dimensional, and second the image of $A$ is closed and has finite codimension. An application of the open mapping theorem shows that the closedness requirement on the image is redundant. 
A well-known example of Fredholm operators due to Riesz: if $C$ is a compact operator then ${\bf 1}-C$ is Fredholm. It is easy to see that the Fredholm property is equivalent to invertiblility modulo finite-rank operators or compact operators.

The main example of such a situation is given by elliptic differential operators acting in sections of vector bundles over compact manifolds. Choosing $\mathcal H_1$ and $\mathcal H_2$ to be appropriate Sobolev spaces of these sections, we find that elliptic operators are Fredholm. 

If $M$ is a noncompact manifold (possibly with boundary) and $A$ an elliptic differential operator on $M$, then it is not necessarily the case that $A$ be Fredholm. That is, the kernel and/or cokernel of $A$ may be infinite-dimensional and/or the image of $A$ may not be closed. In particular, the index as defined above may not be well-defined, but there are notions generalizing the Fredholm property and the index. In this paper we will extend the results in \cite{P1} using one of these generalized Fredholm properties that makes sense when there is a free action of a unimodular Lie group $G$ on $M$. Making appropriate choices of metric on $M$ and in the vector bundles over $M$ and using a Haar measure on $G$, we obtain Hilbert spaces of sections on which the $G$-action is unitary. This action allows us to define a trace ${\rm Tr}_G$ in the algebra of operators commuting with the action of $G$. Restricting this trace to orthogonal projections $P_L$ onto $G$-invariant subspaces $L$ provides a dimension function $\dim_G$ given by 
\[\dim_G(L) = {\rm Tr}_G(P_L).\]
Generalizing the previous definition of the Fredholm property, a $G$-invariant operator $A:\mathcal H_1\to\mathcal H_2$ is said to be $G$-{\it Fredholm} if $\dim_G\ker A<\infty$  and if there exists a closed, invariant subspace $Q\subset {\rm im}(A)$ so that $\dim_G(\mathcal H_2\ominus Q)<\infty$. 

In \cite{P1} it was established that if $M$ is a strongly pseudoconvex complex manifold $M$ with a unimodular Lie group $G$ acting on it freely, by holomorphic transformations so that $M/G$ is compact, then, for $q>0$, the Kohn Laplacian $\square$ is $G$-Fredholm in $L^2(M,\Lambda^{p,q})$.

Here we will relax the requirement on the boundary, requiring only that each point of the boundary have a neighborhood in which a subelliptic estimate holds, as defined below. In addition, we will in detail demonstrate a similar result for the boundary Laplacian $\square_b$. The main results of this paper are.


\begin{theorem}  Let $M$ be a complex manifold with boundary satisfying a local subelliptic property. Let $G$ be a unimodular Lie group acting freely by holomorphic transformations on $M$ so that $M/G$ is compact. Then, for $q>0$, the Kohn Laplacian $\square$ in $L^2(M,\Lambda^{p,q})$ and the boundary Laplacian $\square_b$ in $L^2(bM,\Lambda^{p,q})$ are $G$-Fredholm.\end{theorem}

\begin{rem}{\rm Natural examples of manifolds satisfying the hypotheses of the theorem are the $G$-complexifications of group actions on manifolds as constructed in \cite{HHK}. The unimodularity of $G$ is necessary for the definition of the $G$-Fredholm property.}\end{rem}

The $\dbar$-Neumann problem was proposed by Spencer in the 1950s as a method of obtaining existence theorems for holomorphic functions. Morrey in \cite{Mo} introduced the key {\it basic estimate} and the problem was solved by Kohn in \cite{K1}. We use and develop variants of the techniques in \cite{KN,FK,E} in this work as well as review other approaches to the problem. 

The generalized Fredholm property which we use was first introduced in an abstract setting by M. Breuer \cite{B}. In \cite{GHS}, it is shown that if  $\Gamma$ is taken to be a discrete group and $M$ strongly pseudoconvex, then the Kohn Laplacian $\square$ is $\Gamma$-Fredholm. Note that the natural boundary value problem for $\square$ (called the $\bar\partial$-Neumann problem) is not elliptic, but only subelliptic. In \cite{P1} we extended this result from \cite{GHS} to the situation in which $G$ is a unimodular Lie group, with the manifold still strongly pseudoconvex.

Recent works on covering spaces extending results from \cite{GHS} can be found in \cite{Br1, Br2, Br3, Br4}, some of which are related to the Shafarevich conjecture (which asserts that that the universal covering of a smooth projective variety $X$ is holomorphically convex). More geometric and analytic are \cite{TCM1} dealing with the case in which $M$ is only assumed weakly pseudoconvex. Using cohomological techniques and holomorphic Morse inequalities, the same authors obtain a lower bound for the $\Gamma$-dimension of the space of $L^2$ sections and upper bounds for the $\Gamma$-dimensions of the higher cohomology groups. In \cite{TCM2}, it is shown that the von Neumann dimension of the space of $L^2$ holomorphic sections is bounded below under curvature conditions on $M$.

This work is also connected to recent results in another direction. By Riesz, the $G$-Fredholm property of $\square$ is a natural analogue of the compactness property of the Neumann operator $N$ (the inverse of the restriction of $\square$ to $(\ker\square)^\perp$) in the case that the manifold in question is noncompact. The compactness property of $N$ has been studied extensively in settings in which $M$ is not strongly pseudoconvex. This work is a first step toward obtaining similar results on $G$-manifolds. 

In a series of papers, Catlin, Crist, Fu, McNeal, Straube and others (see the excellent reviews \cite{DK, FS}) have given sufficient conditions for, and obstacles to, compactness of $N$, the Neumann operator, assuming that $M$ is compact. In particular, in this work it is remarked compactness is a local property. Roughly speaking, the $\dbar$-Neumann operator $N_q$ on $\Omega$ is compact if and only if every boundary point has a neighborhood $U$ such that the corresponding $\dbar$-Neumann operator on $U\cap\Omega$ is compact. It would be interesting to see if or in what sense this notion generalizes.

Also, in a more speculative vein, can results analogous to those in \cite{CF} be obtained for the noncompact case by the methods of \cite{SML}?

In section 2 we will introduce the $G$-trace for invariant operators in Hilbert $G$-modules.  Section 3 contains a description of abstract $G$-Fredholm operators and several useful properties. Section 4 treats the relevant results from the theory of the $\bar \partial$-Neumann problem.  In section 5 we prove that $\square$ is $G$-Fredholm and deduce the finite-dimensionality of the reduced Dolbeault cohomology for $q>0$. We also explore some easy consequences of the main theorem regarding the operator $\bar\partial$ on functions. In the appendix there is a derivation of a local {\it a priori} estimate for the Laplacian.


\section{Preliminaries}
    
\subsection{The $G$-Fredholm property} If in a Hilbert space $\mathcal H$ there is a strongly continuous action of a group $R_s:\mathcal H\to\mathcal H$, $(s\in G)$, we denote the space of $G$-equivariant bounded linear operators in $\mathcal H$ by $\mathcal B(\mathcal H)^G$.  In other words $P\in \mathcal B(\mathcal H)^G$ if $P\in{\mathcal B}(\mathcal H)$  and  $R_{s}P=PR_{s}$ for every $s\in G$. Closed, invariant subspaces in a Hilbert space can be obtained as images of projections $P \in \mathcal B(L^2(M))^G$.  Let us restrict our attention to a group acting on itself.
       
For any $s\in G$ define left and right translations $L_s, R_s:L^2(G)\to L^2(G)$ by
    $(L_{s}u)(t) = u(s^{-1}t)$, $(R_su)(t)=u(ts)$.  For $f \in L^{1}(G)$ and $u\in L^{2}(G)$, let  
\[ (L_{f}u)(t) = \int_{G}f(s)(L_{s}u)(t)ds = \int_{G}f(s)u(s^{-1}t)ds.\]
\noindent
The set $\{L_{f}\mid f\in L^{1}(G)\}$ forms an associative algebra of bounded operators 
in $L^2(G)$ which are right-invariant ({\it i.e.} commute with right translations).  
The weak closure of this algebra, $\mathcal L_G\subset \mathcal B(L^{2}(G))$ is a von Neumann algebra by the bicommutant theorem.  We will also need to consider operators $L_{f}$ for $f\in L^{2}(G)$.  These are defined on $C^{\infty}_{c}(G)$ and we may try to extend them by continuity to $L^{2}(G)$. This is not always possible, but we will be concerned only with those $L_{f}$ 
which are bounded or can be extended to bounded linear operators in $L^{2}(G)$. 
It follows from the Schwartz kernel theorem that any bounded right-invariant operator in $L^2(G)$ can be presented in the form $L_f$ for a distribution $f$ on $G$.

We will need the following fact from about group von Neumann algebras ({\it cf.} \cite{P}, sections 5.1 and 7.2). There is a unique trace  ${\rm tr}_{G}$ on $\mathcal L_G \subset\mathcal B (L^2(G))$ agreeing with
\[ {\rm tr}_G (L_{f}^*L_{f}) = \int_G |f(s)|^2 ds, \]    
\noindent
whenever $L_{f}\in \mathcal B (L^{2}(G))$ and $f\in L^2(G).$  Furthermore, 
${\rm tr}_G(A^* A)<\infty$ if and only if there is an $f \in L^2(G)$ for which 
$A= L_{f}\in \mathcal B(L^2(G))$.  
If we define $\tilde f(t) = \bar f(t^{-1})$, and  if $f_{k}, g_{k}\in L^2(G)$, $k=1,\dots, N$, 
then the operator $A= \sum_{1}^N L_{\tilde f_k} L_{g_k}$ is in ${\rm Dom}({\rm tr}_G)$. Furthermore, $A$ takes the form $A=L_h$ with $h$ continuous and ${\rm tr}_G(L_{h}) = h(e)$.

Note that the unimodularity of the group is necessary for the trace property of ${\rm tr}_G$. 
 
Now let us consider free group actions on a manifold. Let $G$ be a Lie group and $G\to M \overset{p}\to X$ be a principal $G$-bundle with compact base $X$.  In particular, this means that we have a free right action of 
    $G$ on $M$ with quotient space $X$, and $p:M\to X$ is the canonical projection.   Having a smooth free action 
 of $G$ on a manifold $M$ with a $G$-invariant measure $d{\bf x}$, 
 and fixing a Haar measure $dt$ on $G$, we obtain a natural quotient measure $dx$ on
 $X=M/G$ which allows us to present $L^2(M)$ in the form
 \[ L^{2}(M)\cong L^{2}(G)\otimes L^{2}(X).\]
It follows that we have
a decomposition of the von Neumann algebra of bounded invariant operators
 \[ \mathcal B(L^{2}(M))^{G}\cong 
 \mathcal B(L^{2}(G))^{G}\otimes \mathcal B(L^{2}(X))\cong 
 \mathcal L_{G}\otimes \mathcal B(L^{2}(X)),\] 
where we have made the identification $\mathcal L_{G}\cong \mathcal B(L^{2}(G))^{G}$.  
In order to measure the invariant subspaces of $L^{2}(M)$, we will use a trace on
$\mathcal L_{G}\otimes \mathcal B(L^{2}(X))$. It is true that there exists a natural 
normal, faithful, semifinite trace on this algebra. It is denoted ${\rm Tr}_G$ and may be constructed as follows.

Let $(\psi_l)_{l\in\mathbb N}$ be an orthonormal basis for $L^2(X)$. Then
  \[  L^2(M) \cong L^{2}(G)\otimes L^2(X) \cong \bigoplus_{l\in \mathbb N} L^2(G)\otimes \psi_{l}.\]
    Denoting by $P_m$ the projection onto the $m^{th}$ summand, we obtain 
a matrix representation of $A\in \mathcal B(L^{2}(M))$ with elements $A_{lm} = P_l A P_m \in \mathcal B(L^{2}(G))$.  If $A\in \mathcal B(L^2(M))^G$, then these matrix elements are invariant operators in $L^{2}(G)$, and so there exist distributions $h_{lm}$ on $G$ so that $A\in \mathcal B(L^{2}(M))^{G}$ has a matrix representation
    \begin{equation}\label{deco}
    A \leftrightarrow [A_{lm}]_{lm}=[L_{h_{lm}}]_{lm}.
    \end{equation}
 For positive $A\in  \mathcal B(L^{2}(M))^{G}$ define 
    \[{\rm Tr}_G (A) = \sum_{l\in \mathbb N} \ {\rm tr}_G (A_{ll}).\]

    The functional ${\rm Tr}_G$ is a normal, faithful, and semifinite trace and is independent of the basis $(\psi_{l})_{l}$ used in its construction, {\it cf.} Section V.2 of \cite{T}. 

 Using this trace, we can define the $G$-dimension of a closed, invariant subspace $L\subset L^2(M)$ as follows. For such a subspace let $P$ be the self-adjoint projection onto $L$. Then $P$ is $G$-invariant and so we may define
 \[\dim_G L ={\rm Tr}_G P. \]
The $G$-dimension has the usual properties of a dimension if it is defined. For example, the $G$-dimension respects the orthogonal sum. Also, if $L_{1}, L_{2}, L$ are closed, invariant subspaces of $L^2(M)$ such that ${\rm dim}_{G}L_{1}>\dim_G(L\ominus L_2)$, then $L_{1}\cap L_{2}\neq \{0\}$ and ${\rm dim}_{G}L_{1}\cap L_{2} \ge {\rm dim}_{G}L_{1}-\dim_{G}(L\ominus L_{2})$. See \cite{GHS}, Lemma (2.1). 

We give now our formal definition of the $G$-Fredholm property.

   \begin{definition} Let $L_0, \ L_1$ be Hilbert spaces on which a unimodular group $G$ acts strongly continuously by unitary transformations, and let $A:L_0 \to L_1$ be a closed densely-defined linear operator commuting with the action of $G$. Such an operator is called $G$-{\it Fredholm} if the following conditions are satisfied:
\begin{itemize}
\item ${\rm dim}_G \ker  A< \infty $
\item there exists a $G$-invariant closed subspace $Q\subset L_1$ so that $Q \subset {\rm im}\ A$ and ${\rm codim}_G \ Q = {\rm dim}_G (L_1 \ominus Q) < \infty .$
\end{itemize}  
\end{definition}
\noindent
Consequences of this definition are collected in \cite{P1, S}.

\subsection{The $\dbar$-Neumann problem} Let $M$ be a complex manifold with boundary. For any integers $p,q$ with $1\leq p,q\leq n$ denote by $C^\infty(M,\Lambda^{p,q})$ the space of all $C^\infty$ forms of type $(p,q)$ on $M$, {\it i.e.} the forms which can be written in local complex coordinates $(z_1, z_2,\dots,z_n)$ as
\begin{equation}\label{pqform}\phi=\sum_{|I|=p,|J|=q}\phi_{I,J} dz^I\wedge d\bar z^J \end{equation}
\noindent
where $dz^I=dz^{i_1}\wedge\dots\wedge dz^{i_p}$, $dz^J=d\bar z^{j_1}\wedge\dots\wedge d\bar z^{j_q}$, $I=(i_1,\dots,i_p)$, $J=(j_1,\dots, j_q)$, $i_1<\dots<i_p$, $j_1<\dots<j_q$,
and the $\phi_{I,J}$ are smooth functions in local coordinates \footnote{These sums will be understood to be over increasing multiindices.}.  For such a form $\phi$, the value of $\dbar\phi$ is
\[\dbar\phi=\sum_{|I|=p,|J|=q} \sum_{k=1}^n
\frac{\partial\phi_{I,J}}{\partial\bar z^k}d\bar z^k\wedge dz^I\wedge d\bar z^J\]
\noindent
so $\dbar = \dbar|_{p,q}$ defines a linear map $\dbar:C^\infty(M,\Lambda^{p,q})\to C^\infty(M,\Lambda^{p,q+1})$. At each $p$, these maps form a complex of vector spaces
\[0 \longrightarrow \Lambda^{p,0}\stackrel{\dbar_{p,0}}{\longrightarrow}\Lambda^{p,1}
\stackrel{\dbar_{p,1}}{\longrightarrow}\dots\stackrel{\dbar_{p,q-1}}{\longrightarrow} \Lambda^{p,q}\stackrel{\dbar_{p,q}}{\longrightarrow}\dots\stackrel{\dbar_{p,n-1}}{\longrightarrow}\Lambda^{p,n}\stackrel{\dbar_{p,n}}{\longrightarrow} 0,\]
and its {\it reduced $L^2$-Dolbeault cohomology spaces} are defined by:
\[L^2\bar H^{p,q}(M)=\ker (\dbar_{p,q})/
\overline{\im (\dbar_{p,q-1})}.\]
\noindent
Since $\ker\dbar$ is a closed subspace in $L^2$, the reduced cohomology space $L^2\bar H^{p,q}(M)$ of our $G$-manifold is a Hilbert space with a strongly continuous action of $G$ by unitary transformations. 

Let us consider $\dbar$ as the maximal operator in $L^2$ and let $\dbar^\ast$ be the Hilbert space adjoint operator (this requires the introduction of boundary conditions). We will deal with several constructions involving $\dbar$ and $\dbar^*$, which we define here. First,  on the pre-domain $\mathcal D^{p,q} = \dom\dbar^*\cap C^\infty(\bar M, \Lambda^{p,q})$, define the quadratic form
\[Q(\phi,\psi) =\langle\dbar\phi,\dbar\psi\rangle + \langle\dbar^*\phi,\dbar^*\psi\rangle + \langle\phi,\psi\rangle.\]
\noindent
The object dominating the frame in this work is Kohn's Laplacian, defined by (on the $(p,q)$-forms)
\[\square=\square_{p,q}=\dbar^\ast\dbar+\dbar\dbar^\ast \quad {\rm on} \quad \dom\square\subset L^2(M,\Lambda^{p,q}).\]
From the quadratic form $Q$ we may construct the operator $F=\dbar\dbar^*+\dbar^*\dbar+1$ on its (and $\square$'s) natural domain, 
\[\dom(F_{p,q}) = \overline{\{\phi\in\mathcal D^{p,q}\mid\dbar\phi\in\mathcal D^{p,q+1}\}} \]
\noindent
where the closure is taken in the $Q$ norm, {\it i.e.} the graph norm of $\square$. The positivity of $\square$ implies that $F$ has a bounded inverse in $L^2(M)$.

\begin{lemma}\label{decomp} The following orthogonal decompositions hold:
\[ L^2(M,\Lambda^{\bullet})= \overline{\im\dbar} \oplus \ker\square \oplus \overline{\im\dbar^*} \qquad \ker\dbar = \overline{\im\dbar} \oplus \ker\square.\]\end{lemma}
\noindent
In particular, we have a $G$-isomorphism of the spaces
\begin{equation}L^2\bar H^{p,q}(M)=\ker\square_{p,q}\end{equation} 
\noindent
and the obvious consequence that if $\square$ is $G$-Fredholm in $L^2(M,\Lambda^{p,q})$, then 
\[\dim_G L^2\bar H^{p,q}(M) = \dim_G \ker\square_{p,q}<\infty.\] 
Similar definitions and their consequences hold for the boundary complex, {\it cf.} Sect. V.4, \cite{FK}. 

\subsection{Sobolev spaces} We will have to describe smoothness of functions, forms, and sections of vector bundles using $G$-invariant Sobolev spaces which we describe here. The $G$ action induces an invariant Riemannian metric on $M$ so that with respect to this structure $M$ has bounded geometry. As in \cite{G, S1} we may construct appropriate partitions of unity and, with local geodesic coordinates, assemble $G$-invariant integer Sobolev spaces $H^s(M)$. Fractional Sobolev spaces may be constructed by functional-analytic means or more specifically by interpolation.

Norms corresponding to negative Sobolev spaces are defined as usual;
\[\|\phi\|_{H^{-s}(M)} = \sup\left\{\frac{|\langle\phi,\psi\rangle|}{\|\psi\|_{H^s(M)}}\mid \psi\in C^\infty_c(M)\right\}\qquad (s>0).\]

If $E$ is a vector $G$-bundle over $M$, then we may introduce a $G$-invariant inner product structure on $E$. Together with a $G$-invariant measure on $M$, we define the Hilbert spaces of sections of $E$ which we denote $H^s(M,E)$, for $s=0,1,2,\dots$. Because $X=M/G$ is compact, the spaces $H^s(M,E)$ do not depend on the choices of invariant metric on $M$ or of invariant inner product on $E$. Note that, in particular, spaces of sections in natural tensor bundles on a $G$-manifold have natural, invariant Sobolev structures.

We will also need anisotropic Sobolev spaces in neighborhoods of boundary points. Assume that our complex manifold $M$ has nonempty, smooth boundary $bM$,  $\bar M=M\cup bM$, so that $M$ is the interior of $\bar M$, and ${\rm dim}_{\mathbb C}(M)=n$. Also assume that $\bar M\subset\tilde M$, where $\tilde M$ is a complex neighborhood of $\bar M$ of the same dimension, such that $bM$ is in the interior of $\tilde M$. 

If $U$ is a sufficiently small neighborhood of a point $z\in bM$, then there exist real coordinates $x=(t_1,\dots,t_{2n-1},\rho)$ in $U$ for which $z\leftrightarrow 0$, the set $U\cap M$ corresponds to $\{x\mid \rho<0\}$, and $bM$ corresponds to $\{x\mid \rho=0\}$, {\it i.e.} $\rho$ is the defining function. Coordinates such as these in $U$ are called a {\it special boundary chart}. By means of the Fourier transform in directions tangential to the boundary, it is possible to define tangential Sobolev norms $|||\cdot|||_s$, of arbitrary real order $s$, as follows, {\it cf.} Sect. II.4 of \cite{FK}.
  
With the tangential Fourier transform in a special boundary chart
\[\tilde u(\tau,\rho)=\frac{1}{(2\pi)^{(2n-1)/2}}\int_{\mathbb R^{2n-1}}dt\ e^{-i\langle t,\tau\rangle}u(t,\rho),\]
\noindent
define for $s\in\mathbb R$, the operators 
\[\Lambda_{\bf t}^s u(t,\rho)=\frac{1}{(2\pi)^{(2n-1)/2}}\int_{\mathbb R^{2n-1}}d\tau \ e^{i\langle t,\tau\rangle}(1+|\tau|^2)^{s/2}\tilde u(\tau,\rho)\]
\noindent
(${\bf t}$ means tangential) and define the tangential Sobolev norms by
\[|||u|||_s^2 = \int_{\mathbb R^{2n-1}}d\tau\ \int_{-\infty}^0 d\rho\ (1+|\tau|^2)^s|\tilde u(\tau,\rho)|^2 .\]
\noindent
With $D^j=D^j_{\bf t}=\frac{1}{i}\frac{\partial}{\partial t_j}$ for $j=1,\dots,2n-1$ the derivatives in tangential directions and $D^{2n} = D_\rho$, define the norms
\begin{equation}\label{tangnorms}|||Du|||_s^2 = \sum_1^{2n} |||D^j u|||_s^2 + |||u|||_s^2 \approx  |||u|||_{s+1}^2 + |||D_\rho u|||_s^2.\end{equation}
For two norms $\|\cdot\|$ and $|\cdot|$, we write $|\phi|\lesssim\|\phi\|$ to mean that there exists a constant $C>0$ such that $|\phi|\le C\|\phi\|$ for $\phi$ in whatever set relevant to the context. Similarly, we will write $|\phi|\approx\|\phi\|$ to mean that $|\phi|\lesssim\|\phi\|$ and $\|\phi\|\lesssim|\phi|$. 

\section{Subelliptic Estimates}

As in \cite{K2}, if $x\in\bar M$, we say that the $\dbar$-Neumann problem for $(p,q)$-forms satisfies a {\it subelliptic estimate of order} $\epsilon$ at $x$ if there exists a neighborhood $V$ of $x$ and a constant $C>0$ such that  
\begin{equation}\label{subell}\| \phi \|_{\epsilon}^2 \le C(\|\dbar \phi\|^2 + \|\dbar^* \phi\|^2 + \| \phi\|^2 ),\end{equation}
\noindent
uniformly for $\phi\in C^\infty(V,\Lambda^{p,q})\cap\dom\dbar\cap\dom\dbar^*$. 

Denote by $\mathcal E^q(\epsilon)$ the set of points $x$ of $\bar M$ for which there exists a neighborhood $U$ of $x$ on which a subelliptic estimate of order $\epsilon$ holds. 

\begin{definition}We will say that a complex manifold $M$ satisfies a  subelliptic estimate of order $\epsilon$ if each point of the boundary does.\end{definition}

Our setting will be a complex manifold $M$ with nonempty boundary that is also the total space of a $G$-bundle
\[G\longrightarrow M\longrightarrow X\]
\noindent
with $\bar X= \bar M/G$ compact, and satisfying the condition of the definition.

\bigskip

We will provide here an important fact concerning the relation between tangential and normal differentiability of functions in an elliptic system. 

\begin{lemma}\label{k2sthing}(Prop. 3.10, \cite{K2}) Let $\epsilon>0$ and $x\in bM$. Then $x\in\mathcal E^q(\epsilon)$ if and only if there exists a neighborhood $V'$ of $x$ and a constant $C'>0$ such that 
\begin{equation}\label{subell2}|||\phi |||_{\epsilon}^2 \le C'(\|\dbar \phi\|^2 + \|\dbar^* \phi\|^2 + \| \phi\|^2)\qquad (\phi\in\mathcal D^{p,q}_{V'}).\end{equation}\end{lemma}
\begin{proof} This proposition is a consequence of the fact that $bM$ is noncharacteristic with respect to $Q$ (which is elliptic). See the proofs of Thm. 2.4.8 and the expression 2.3.5 in \cite{FK} or \cite{CS}, Lemmata 5.2.3--4. Roughly speaking, the ellipticity of $\square$ allows one to write the derivatives of $\phi$ in the direction normal to the boundary in terms of the tangential derivatives and other terms in $\square\phi$. The tangential derivatives of $\phi$ are of course controlled by the tangential norms, $|||\cdot|||_\epsilon$ and thus so are the normal ones. Note also that if we vary $\epsilon$, the size of the neighborhood $V$ in which the subelliptic estimate holds need not change; the constant can be changed. \end{proof}

\subsection{The basic estimate} For the following definitions, the reader is referred to \cite{FK} for details. Let $\Pi_{p,q}$ be the projection onto the space of forms of type $(p,q)$ as in Eq. \eqref{pqform}. For each $p\in bM$, the {\it Levi form} at $p$ is the Hermitian form on the $(n-1)$-dimensional space $(\Pi_{1,0}\mathbb C T_p M)\cap \mathbb C T_p bM$ given by 
\[(L_1,L_2)\mapsto\langle\partial\dbar\rho, L_1\wedge\bar L_2\rangle.\]
 A complex manifold $M$ satisfies {\it condition} $Z(q)$ if the Levi form has at least $n-q$ positive eigenvalues or at least $q+1$ negative eigenvalues at each point $p\in bM$. A strongly pseudoconvex complex manifold satisfies properties $Z(q)$ for $q=1,2,\dots,n$.
 
Let $\omega_1,\dots,\omega_n$ be a local orthonormal basis for $C^\infty(\bar M, \Lambda^{1,0})$ on the patch $U\subset\tilde M$ and let $L_1,\dots,L_n$ be dual vector fields. Then if $\phi\in C^\infty(\bar M, \Lambda^{1,0})$ has support in $U$, we may write $\phi=\sum_{IJ}\phi_{IJ}\omega^I\wedge\bar\omega^J$. In $U$ define the norm $E_U$ with
\[E_U(\phi)^2 = \sum_{IJk} \|\bar L_k \phi_{IJ} \|^2 + \int_{bM}|\phi|^2 +\|\phi\|^2.\] 
\noindent
Again by a partition of unity argument as in \cite{G, S1}, one may sum these local definitions and obtain a global norm via $E(\phi)^2 =\sum_j E_{U_j}(\phi)^2$.

That {\it the basic estimate is satisfied in} $\mathcal D^{p,q}$ means that there exists a $C>0$ such that $E(\phi)^2\le C Q(\phi,\phi)$ uniformly for $\phi\in\mathcal D^{p,q}$. 

H\"ormander in \cite{H} showed that the condition $Z(q)$ is equivalent to the basic estimate. We will need only the forward implication, proven in the compact case in \cite{FK}, Theorem 3.2.10. The proof there relies on the compactness of $M$ but we will need only local estimates, as has been exploited in \cite{E}, and uniformity as guaranteed in our case by the structure of $M$, a $G$-manifold with compact quotient $M/G$. When we say ``Suppose the basic estimate holds in $\mathcal D^{p,q}$,'' we will mean that the estimate above holds locally. 

Folland and Kohn in Theorem 2.4.4 prove (a stronger version of) the following:

\begin{theorem}\label{2.4.4} For every $p\in bM$ there is a (small) special boundary chart $V$ containing $p$ such that 
\[|||D\phi|||^2_{-1/2}\lesssim E(\phi)^2.\]
\end{theorem}

It is here that one introduces the basic estimate (b.e.) as follows, 
\[|||D\phi|||^2_{-1/2}\lesssim E(\phi)^2\stackrel {\tiny\rm b.e.} \lesssim Q(\phi,\phi),\]
\noindent
yielding the local subelliptic estimate 
\[||| D\phi |||_{\epsilon-1}^2 \le C(\|\dbar \phi\|^2 + \|\dbar^* \phi\|^2 + \|\phi\|^2 ),\]
of order $\epsilon=1/2$ in a small special boundary chart $V$. For noncompact $M$ we will follow \cite{E,P1} and only assume that the basic estimate holds locally. 


\subsection{Finite type} A subelliptic estimate holds for the $\dbar$-Neumann problem near a point in the boundary of a pseudoconvex domain if and only if the D'Angelo type at the point is finite, \cite{C1,D}. In \cite{K2} the condition of ``finite ideal type'' is shown to be sufficient for subellipticity. In \cite{K3}, Kohn conjectures this condition to be equivalent to the notion of finite type due to D'Angelo. Another property, called {\it property} $(P)$ ({\it cf.} \cite{C2}), as well as McNeal's {\it property} $(\tilde P)$ \cite{Mc}, quantified in Herbig \cite{He}, give rise to subelliptic estimates. 
Michel and Shaw \cite{MS} and Henkin, Iordan, Kohn, \cite{HIK}

Straube achieved in \cite{St} the removal of the assumption of smoothness of the boundary for property $(P)$. A collection of properties on the boundary guaranteeing subelliptic estimates are collected in \cite{FS}. 


\subsection{Other criteria} The results of \cite{K2, DF} imply that a subelliptic estimate on $(p,q)$-forms is equivalent to the absence of germs of $q$-dimensional complex varieties in the boundary. Furthermore, these estimates always hold on any bounded pseudoconvex domain with real-analytic boundary.

\section{The $\dbar$-Neumann Problem}

\subsection{Interior Estimates}\label{intell} First, modify $Q$ by adding a Sobolev 1-norm:
\[Q^\delta(\phi,\phi) = Q(\phi,\phi) + \delta \|\phi\|_{H^1(M)}^2 \quad q\in (0,1].\]
It is obvious that for each $\delta>0$ we have that $\|\phi\|_1\lesssim Q^\delta(\phi,\phi)$, {\it i.e.} G\aa rding's estimate holds, as long as $\phi$ is in the natural domain of $Q^\delta$, defined as for $Q$. It follows that if $F^\delta$ is the operator induced by $Q^\delta$, (Prop. 1.3.3, \cite{FK}) then $F^\delta$ has regularity properties similar to those of $F$ in the interior of $M$, that is a (genuine) gain of two derivatives, {\it cf.} Theorem 2.2.9, \cite{FK}. Estimates of this type will not be used in the regularization procedure as they depend strongly on $\delta>0$.

Instead, one shows that a genuine estimate for the $Q^\delta$ gaining one derivative can be made to hold {\it uniformly in} $\delta$. This is Theorem 2.3.4 of \cite{FK} and it is of the following form. For $\delta>0$ let $\phi^\delta$ be the unique solution of $F^\delta\phi^\delta = \alpha$. Then the following estimates hold,
\begin{equation}\label{intest}\|\zeta\phi^\delta\|_{s+1}^2\le C_k \|\zeta_1 \alpha\|_s^2 + \|\alpha\|^2.\end{equation}
In particular, if $\alpha$ is smooth, then $\phi^\delta$ is smooth and the constants $C_k$ do not depend on $\delta$. 

\subsection{First boundary a priori estimate} 
\begin{lemma} Suppose that in $\mathcal D^{p,q}$ a subelliptic estimate of order $\epsilon>0$ holds in a special boundary chart $V$. Let $\{\zeta_k\}_0^\infty$ be a sequence of real functions in $\Lambda^{0,0}_0(V\cap \bar M)$ such that $\zeta_k=1$ on ${\rm supp}\ \zeta_{k+1}$.  Then for each integer $k\ge1$
\begin{equation}\label{tangresult}|||D\zeta_k \phi |||_{k\epsilon-1}^2 \lesssim |||\zeta_1 F\phi |||_{(k-2)\epsilon}^2 + \|F\phi\|^2.\end{equation}
\end{lemma}
\begin{proof} An application of Lemma \ref{k2sthing} converts the assumption to  
\begin{equation}\label{workingdef}|||D\zeta\phi|||_{\epsilon-1}\lesssim Q(\zeta_1\phi,\zeta_1\phi).\end{equation}
\noindent
In Sect. 6.4 of \cite{KN}, it is shown that \eqref{workingdef} implies the tangential estimates \eqref{tangresult}.\end{proof}

\begin{corollary}\label{englis3.1} (\cite{E}, Eq. 3.1) Suppose that in $\mathcal D^{p,q}$ a subelliptic estimate of order $\epsilon>0$ holds in a special boundary chart $V$. Then we obtain {\it a priori} estimates for each integer $k\ge1$,
\begin{equation}\label{firstap}\|\zeta\phi\|_{k\epsilon}^2\lesssim \|\zeta_1 F\phi \|_{(k-2)\epsilon}^2 + \|F\phi\|^2, \quad(\phi\in C^\infty(V,\Lambda^{p,q})\cap\dom F).\end{equation}\end{corollary}
\begin{proof} The estimate \eqref{tangresult} implies the result because $||D\zeta_k \phi ||_{k\epsilon-1} \approx |||D\zeta_k \phi |||_{k\epsilon-1}$ by Lemma \ref{k2sthing} and $||D\zeta_k \phi ||_{k\epsilon-1} \approx||\zeta_k \phi ||_{k\epsilon}$ obviously. Next, $|||\zeta_1 F\phi |||_{(k-2)\epsilon} \lesssim \|\zeta_1 F\phi \|_{(k-2)\epsilon}$ since the smaller norm doesn't differentiate in the normal direction and the tangential directions' order is the same.\end{proof}

These boundary estimates can be summed with the interior estimates \eqref{intest} (invoking uniformity) so that \eqref{firstap} holds for $\phi$ with support inside $M$, thus for each integer $k\ge1$,
\begin{equation}\label{bigap1}\|\zeta\phi\|_{k\epsilon}^2\lesssim \|\zeta_1 F\phi \|_{(k-2)\epsilon}^2 + \|F\phi\|^2, \quad(\phi\in C^\infty(\bar M,\Lambda^{p,q})\cap\dom F).\end{equation}
%

\subsection{Second boundary a priori estimate} This subsection is a small modification of the appendix in \cite{P1} and Sect. 6.4 of \cite{KN}. Our goal is a version of the $\epsilon=1/2$ estimate, Lemma 7.9 \cite{P1},
\[|||D\zeta_k \phi|||^2_{(k-2)/2}\lesssim \|\zeta_0F \phi\|^2_{(k-2)/2}+\|\zeta_0 \phi\|^2, \quad (\phi\in C^\infty(M,\Lambda^{p,q})\cap\dom F).\]
The interior estimates from subsection \ref{intell} will be used unchanged here. 

We will systematically label sequences of real-valued, cutoff functions $(\zeta_k)_k\subset C^\infty_c(M)$ such that $\zeta_{k}|_{{\rm supp}(\zeta_{k+1})}=1$ for $k=0,1,2,\dots$.
\begin{lemma}\label{cutFolland} Let $r>0$, $U$ be a special boundary chart and let $\zeta,\zeta_0,\zeta_1$ be real-valued functions in $C^\infty_c(U)$ with $\zeta_1=1$ on ${\rm supp}(\zeta)$ and $\zeta_0=1$ on ${\rm supp}(\zeta_1)$. Then for $A=\zeta_1\Lambda_{\bf t}^r\zeta$ and for $A'$ the formal adjoint of $A$ with respect to the inner product on $L^2(M)$,
\[ Q(A \phi,A \phi)-\mathfrak{Re}\ Q( \phi,\zeta_0 A'A \phi) = \mathcal O(|||D\zeta_0 \phi|||_{r-1}^2)\]
\[Q(\zeta \phi,\zeta \phi)-\mathfrak{Re}\ Q( \phi,\zeta_0\zeta^2 \phi) = \mathcal O(\|\zeta_0 \phi\|^2),\]
\noindent
uniformly for $\phi\in\mathcal D^{p,q}\cap\Lambda_0^{p,q}(U\cap\bar M)$.\end{lemma}
\begin{proof}These are simple consequences of the fact that the domain $\mathcal D^{p,q}$ of $\vartheta$ is preserved under the application of a cutoff function ({\it cf.} 2.3.2 of \cite{FK}) and lemmata 2.4.2 and 2.4.3 of \cite{FK} applied to $\zeta_0 u$. \end{proof}

\noindent
If we assume further that $\phi\in {\rm Dom}(F)$, ({\it cf.} \cite{FK}, Prop. 1.3.5) we may write
\begin{align}\label{locs} Q(A \phi,A \phi)-\mathfrak{Re}\ \langle\zeta_0F \phi,A'A \phi\rangle = \mathcal O(|||D\zeta_0 \phi|||_{r-1}^2)\\
Q(\zeta \phi,\zeta \phi)-\mathfrak{Re}\ \langle\zeta_0F \phi,\zeta^2 \phi\rangle = \mathcal O(\|\zeta_0 \phi\|^2).\end{align}

The following is our replacement of Lemma 7.9 of \cite{P1}, utilizing the subelliptic estimate, {\it cf.} \cite{KN}, p 472, \cite{FK}, Lemma 2.4.6. Note the similarity to Prop 3.1.11, \cite{FK}, proven with the aid of compactness.
 

\begin{lemma}\label{biglem} Suppose that in $\mathcal D^{p,q}$ a subelliptic estimate of order $\epsilon>0$ holds in a special boundary chart $V$. Let $\{\zeta_k\}_0^\infty$ be a sequence of real functions in $\Lambda^{0,0}_0(V\cap \bar M)$ such that $\zeta_k=1$ on ${\rm supp}\ \zeta_{k+1}$.  Then for each integer $k\ge1$, 
\begin{equation}\label{aag}\|\zeta_k \phi\|^2_{k\epsilon}\lesssim \|\zeta_0F \phi\|^2_{(k-2)\epsilon}+\|\zeta_0 \phi\|^2.\end{equation}
\noindent
uniformly for $\phi\in {\rm Dom}(F)\cap\mathcal D^{p,q}$.\end{lemma}

\begin{proof}Assuming the subelliptic estimate and noting that multiplication by $\zeta_1$ preserves $\mathcal D^{p,q}$, we have
\[\|\zeta_1 \phi\|_{\epsilon}^2\lesssim Q(\zeta_1 \phi,\zeta_1 \phi),\qquad  \phi\in\mathcal D^{p,q}\cap \Lambda_0^{p,q}(V\cap\bar M).\]

\noindent
If we insert a real-valued cutoff function $\zeta_0$ equal 1 on the support of $\zeta_1$ and apply Lemma \eqref{cutFolland}, to the form $\zeta_0 u$ we have 
\[\|\zeta_1 \phi\|_{\epsilon}^2\lesssim \mathfrak{Re}\ Q( \phi,\zeta_0\zeta_1^2 \phi)+\mathcal O(\|\zeta_0 \phi\|^2).\]
\[ =  \mathfrak{Re}\  \langle F \phi,\zeta_1^2 \phi\rangle+\mathcal O(\|\zeta_0 \phi\|^2) =  \mathfrak{Re}\  \langle\zeta_1 F \phi,\zeta_1 \phi\rangle+\mathcal O(\|\zeta_0 \phi\|^2).\]

\noindent
Now, by the generalized Schwartz inequality, we have
\[\|\zeta_1 \phi\|_{\epsilon}^2\lesssim  \mathfrak{Re}\  \langle \zeta_1F \phi,\zeta_1 \phi\rangle+\mathcal O(\|\zeta_0 \phi\|^2)\lesssim \|\zeta_1F \phi\|_{-\epsilon}\|\zeta_1 \phi\|_{\epsilon}+\mathcal O(\|\zeta_0 \phi\|^2).\]
\noindent
Now, for any $c>0$ there exists a $C>0$ sufficiently large so that
\[\|\zeta_1 \phi\|_{\epsilon}^2\lesssim C\|\zeta_1F \phi\|_{-\epsilon}^2+c\|\zeta_1 \phi\|_{\epsilon}^2+\mathcal O(\|\zeta_0 \phi\|^2)\]
\noindent
thus
\[\|\zeta_1 \phi\|_{\epsilon}^2\lesssim \|\zeta_0F \phi\|_{-\epsilon}^2+\|\zeta_0 \phi\|^2,\]
\noindent
and we have shown that the lemma is true for $k=1$. Assume the lemma true for $k-1$, {\it i.e.} 
\begin{equation}\label{indhyp}\|\zeta_{k-1} \phi\|^2_{(k-1)\epsilon}\lesssim \|\zeta_0 F \phi\|^2_{(k-3)\epsilon} +\|\zeta_0 \phi\|^2. \end{equation}
\noindent
We follow and modify where necessary the proof in \cite{KN}, citing intermediate results.  According to Lemma \ref{k2sthing}, we may make the second replacement in our estimates, 
\[||\zeta_{k-1} \phi ||_{(k-1)\epsilon} \approx ||D\zeta_{k-1} \phi ||_{(k-1)\epsilon-1} \leftrightarrow |||D\zeta_{k-1} \phi |||_{(k-1)\epsilon-1};\] 
\noindent
the first is obvious. Abbreviating $\Lambda_{\bf t}^{(k-1)\epsilon}=\Lambda$ and $A=\zeta_1\Lambda\zeta_k$, Kohn-Nirenberg derive their (6.6), 
\begin{equation}\label{commpush1}|||D\zeta_k \phi|||^2_{k\epsilon-1}\lesssim |||DA\phi|||^2_{\epsilon-1} + |||D\zeta_{k-1} \phi|||^2_{(k-1)\epsilon-1}.\end{equation}
\noindent
Next, KN arrive at 
\[|||DA\phi|||^2_{\epsilon-1}\lesssim |||\zeta_1F\phi|||_{(k-2)\epsilon}|||D\zeta_k  \phi|||_{k\epsilon-1}+ |||D\zeta_{k-1}\phi|||_{(k-1)\epsilon-1}^2 \]
\noindent
which they substitute into the estimate \eqref{commpush1}, obtaining
\[|||D\zeta_k \phi|||^2_{k\epsilon-1}\lesssim |||\zeta_1F\phi|||_{(k-2)\epsilon}|||D\zeta_k  \phi|||_{k\epsilon-1}+ |||D\zeta_{k-1} \phi|||^2_{(k-1)\epsilon-1}\]
\[\lesssim C |||\zeta_1F\phi|||_{(k-2)\epsilon}^2+c |||D\zeta_k  \phi|||_{k\epsilon-1}^2+ |||D\zeta_{k-1} \phi|||^2_{(k-1)\epsilon-1}.\]
\[\lesssim  |||\zeta_1F\phi|||_{(k-2)\epsilon}^2 + |||D\zeta_{k-1} \phi|||^2_{(k-1)\epsilon-1}.\]
\noindent
The induction hypothesis reads 
\[|||D\zeta_{k-1} \phi |||_{(k-1)\epsilon-1}\lesssim \|\zeta_0 F \phi\|^2_{(k-3)\epsilon} +\|\zeta_0 \phi\|^2\]
\noindent
so 
\[|||D\zeta_k \phi|||^2_{k\epsilon-1}\lesssim  |||\zeta_1F\phi|||_{(k-2)\epsilon}^2 +\|\zeta_0 F \phi\|^2_{(k-3)\epsilon} +\|\zeta_0 \phi\|^2.\]
\noindent
Since $ |||\zeta_1F\phi|||_{(k-2)\epsilon}\le  \|\zeta_1F\phi\|_{(k-2)\epsilon}$ and $\|\zeta_0 F \phi\|^2_{(k-3)\epsilon}\le \|\zeta_1F\phi\|_{(k-2)\epsilon}$, we have that 
\[|||D\zeta_k \phi|||^2_{k\epsilon-1}\lesssim \|\zeta_0 F \phi\|^2_{(k-2)\epsilon} +\|\zeta_0 \phi\|^2.\]
\noindent 
Again invoking Lemma \ref{k2sthing}, we replace $|||D\zeta_k \phi|||_{k\epsilon-1}$ with $\|\zeta_k \phi\|_{k\epsilon}$ and obtain the result. \end{proof}
%

The following is Thm. 4.5 from \cite{P1}, adapted to our current situation.

\begin{theorem}\label{soboreg} Let $q>0$ and suppose a subelliptic estimate of order $\epsilon$ holds in a $G$-manifold $M$ with compact orbit space $M/G$ and invariant structures. Then, for every integer $k\ge2$ we have an estimate 
\[\|\phi\|_{k\epsilon}^2\lesssim \|\square\phi\|_{(k-2)\epsilon}^2+ \|\phi\|^2,\]
\noindent
holding uniformly for $\phi\in {\rm Dom}(\square)\cap C^\infty(\bar M,\Lambda^{p,q})$.\end{theorem}
\begin{proof} From Lemma \ref{biglem}, uniformity of the group invariance, the compactness of $M/G$, and again as in \cite{G, S1} we may construct appropriate partitions of unity (of bounded multiplicity) and glue together the local {\it a priori} estimates \eqref{aag} to obtain the global estimate. \end{proof}

\subsection{Genuine estimates} So far, all that we have shown assumes that $\phi$ is a smooth solution of $F\phi=\alpha$ and derives estimates involving the Sobolev norms of $\phi$ in terms of those of $\alpha$. It remains to show that if $\phi=F^{-1}\alpha$ now is the element of $L^2$ guaranteed by the bounded invertibility of $F$, that this $\phi$ has the same smoothness properties as predicted by the {\it a prioris}. {\it I.e.} it is not yet clear that these derivatives of $\phi$ exist. There are several ways to deduce this, {\it e.g.} in \S 4.1--2 of \cite{KN} who perform an {\it elliptic regularization}.\footnote{In \cite{FK}, \S II.5, the same method as in \cite{KN} is used, but the theorem's validity is restricted to the case of compact $M$ as its proof uses Rellich's lemma. In \cite{E}, \S 3 a different method is proposed, valid, like ours, in the case of $M$ noncompact.} The upshot is that from the first {\it a priori} estimate it follows that 

\begin{theorem} Let $\alpha\in L^2(M,\Lambda^{p,q})$ and a subelliptic estimate of order $\epsilon$ hold in $\mathcal D^{p,q}$ Let $U$ be an open subset of $\bar M$ with compact closure, and $\zeta, \zeta_{1}\in C^{\infty}_{c}(U)$ with $\zeta_{1}|_{{\rm supp}(\zeta)}=1$.  If $q>0$, $j$ is a nonnegative integer, and $\alpha|_{U}\in C^\infty(U,\Lambda^{p,q})$, then $\zeta(\square +1)^{-1}\alpha\in C^\infty(M,\Lambda^{p,q})$ and
\begin{equation}\label{realest}\|\zeta(\square +1)^{-1}\alpha\|_{(j+2)\epsilon}^2\lesssim \|\zeta_1 \alpha \|_{j\epsilon}^2 + \|\alpha\|^2,\end{equation}\end{theorem}
\noindent
uniformly for $\alpha$.

\begin{corollary}\label{globalize} Let $M$ satisfy a subelliptic estimate of order $\epsilon>0$ uniformly, let $U$ be an open subset of $\bar M$ with compact closure, and $\zeta, \zeta_{1}\in C^{\infty}_{c}(U)$ for which $\zeta_{1}|_{{\rm supp}(\zeta)}=1$.  If $q>0$, $j$ is a nonnegative integer, and $\alpha|_{U}\in H^{j\epsilon}(U,\Lambda^{p,q})$, then $\zeta(\square +1)^{-1}\alpha\in H^{(j+2)\epsilon}(M,\Lambda^{p,q})$ and there exist constants $C_j>0$ so that 
 \begin{equation}\label{prima}\|\zeta (\square +1)^{-1}\alpha\|_{(j+2)\epsilon}^2\le C_k(\|\zeta_{1}\alpha\|_{j\epsilon}^2+\|\alpha\|_0^2).\end{equation}
 \end{corollary}
 \begin{proof}This is an extension of Prop. 3.1.1 from \cite{FK}. It is a density argument applied to the real estimates above.\end{proof}

The following is a modified version of Cor. 4.3 from \cite{P1}.
\begin{corollary}\label{esso}Let $q>0$ and $\square=\int_{0}^{\infty}\lambda dE_{\lambda}$ be the spectral decomposition of the Laplacian in $L^{2}(M,\Lambda^{p,q})$.  If $\delta>0$ and $P=\int_{0}^{\delta}dE_{\lambda}$ then $\im P\subset C^{\infty}(\bar M,\Lambda^{p,q})$.\end{corollary}
\begin{proof}We show that $\im P\subset H^{s}_{\rm loc}(M,\Lambda^{p,q})$ for all $s>0$ and invoke the Sobolev embedding theorem.  Let $U, \zeta, \zeta_{1}$ be as in the previous theorem.  Since $\im((\square +1)^{-1})=\dom(\square)$, the following is true. For every $\phi\in\dom(\square)$ with $\alpha=\square\phi +\phi\in H^{j\epsilon}_{\rm loc}(M)$, we have $\phi\in H^{(j+2)\epsilon}_{\rm loc}(M)$ and the estimate \eqref{prima} holds uniformly. Let $\phi\in\im P$.  Applying the theorem with $j=0$, we have $\im P\subset H^{2\epsilon}_{\rm loc}(M,\Lambda^{p,q})$.  Now assume $\phi\in\im P\subset H^{j\epsilon}_{\rm loc}(M,\Lambda^{p,q})$.  Then
\[(\square+1)\phi=(\square+1)P\phi=P(\square+1)\phi\in H^{j\epsilon}_{\rm loc}(M,\Lambda^{p,q}).\]
\noindent  
We conclude that $\phi\in H^{(j+2)\epsilon}_{\rm loc}(M,\Lambda^{p,q})$ and so $\im P\subset H^{(j+2)\epsilon}_{\rm loc}(M,\Lambda^{p,q})$.\end{proof}

\begin{corollary}\label{bomb}Let $q>0$ and $\square=\int_{0}^{\infty}\lambda dE_{\lambda}$ be the spectral decomposition of the Laplacian in $L^{2}(M,\Lambda^{p,q})$.  If $\delta>0$ and $P=\int_{0}^{\delta}dE_{\lambda}$ then $\im P\subset H^{\infty}(M,\Lambda^{p,q})$.\end{corollary}
\begin{proof} Let $\phi\in\im P$. By Cor. \ref{esso}, $\phi\in C^\infty(\bar M,\Lambda^{p,q})$ and so Thm. \ref{soboreg} holds. Since $\phi\in\dom\square^k$, $k=1,2,3,\dots$, we have
\[\|\square^{k-j}\phi\|_{(j+2)\epsilon}\lesssim \|\square^{k-j+1} \phi\|_{j\epsilon} + \|\square^{k-j}\phi\|_0 \quad (j=1,2,\dots,k).\]
\end{proof}

\noindent
We need the following fact regarding Sobolev spaces on manifolds with boundary. For $s>0$, denote by $H^{-s}(\bar M)$ the dual space of $H^s(\bar M)$. {\it I.e.} $H^{-s}(\bar M)=(H^s(\bar M))'$. Remark 12.5 of \cite{LM} gives that if $M$ is a manifold with boundary and $s>0$, then $H^{-s}(\bar M)$ consists of elements of $H^{-s}(\tilde M)$ whose support is in $\bar M$.

\begin{corollary}\label{applied}Let $q>0$ and $\square=\int_{0}^{\infty}\lambda dE_{\lambda}$ be the spectral decomposition of the Laplacian in $L^{2}(M,\Lambda^{p,q})$.  If $\delta>0$ and $P=\int_{0}^{\delta}dE_{\lambda}$ then $P: H^{-s}(\bar M,\Lambda^{p,q})\to H^{s}(M,\Lambda^{p,q})$ for any positive integer $s$.\end{corollary} 
\begin{proof} In Lemma \ref{bomb} we established that spectral projections $P$ of $\square$ take $L^2(M)$ to $H^s(M)$ for all $s>0$. It follows that $P$ can be extended so that  $P:H^{-s}(\bar M)\to L^2(M)$. Since $P^2=P$ on $H^\infty(M)\subset L^2(M)$, a dense subspace of all the $H^s(\bar M)$, $(s\in\mathbb R)$ we conclude that $P:H^{-s}(\bar M)\to H^s(M)$ for all $s>0$. \end{proof}

             \section{The $G$-Fredholm Property of $\square$}

\noindent
We will need a description of $G$-operators in terms of their Schwartz kernels.  If $P\in \mathcal B(L^{2}(M))^{G}$, its kernel $K_{P}$ satisfies
\[ K_P({\bf x},{\bf y})= K_P({\bf x}t,{\bf y}t),\quad t\in G.\]
\noindent
Thus $K_{P}$ descends to a distribution on the quotient $\frac{M\times M}{G}$.  The measure taken on $\frac{M\times M}{G}$ is simply the quotient measure.
 
\begin{lemma}\label{kersmooth} If $P: L^2(M) \to H^{\infty}(M)$ is a self-adjoint projection, then its Schwartz kernel $K_P$ is smooth. \end{lemma}
 \begin{proof} Since ${\bf y}\mapsto\delta_{\bf y}$ is a smooth function on $\bar M$ with values in $H^{-\infty}_{c}(\bar M)$, the composition
  \[({\bf x},{\bf y})\longmapsto (P\delta_{\bf y})({\bf x}) = \int_{M} K_P({\bf x},{\bf z})\delta_{\bf y}({\bf z})d{\bf z} = K_P({\bf x},{\bf y})\]
 \noindent 
is jointly smooth.  \end{proof} 

 \begin{lemma}\label{kerint} If $P\in \mathcal B(L^{2}(M))^{G}$ is a self-adjoint, invariant projection so that ${\rm im}(P)\subset H^{\infty}(M)$, then $K_{P}\in L^{2}(\frac{M\times M}{G})$.\end{lemma}
 \begin{proof}Fix ${\bf x}\in \bar M$.  If $P:L^{2}(M)\to C^{\infty}(\bar M)$, the closed graph theorem applied to $P$ implies $u\in L^{2}(M)\mapsto (Pu)({\bf x}) \in \mathbb C$ is a bounded linear functional.  The Riesz representation theorem then gives that there exists a function $h_{\bf x}\in L^{2}(M)$ so that
\[(Pu)({\bf x}) = \langle h_{\bf x},u\rangle \quad u\in L^{2}(M).\]
Since $(Pu)({\bf x})=\int_{M}K_{P}({\bf x},{\bf y})u({\bf y})d{\bf y}$, and agrees with $\langle h_{\bf x},u\rangle$ when $u$ has compact support, $h_{\bf x}=K_{P}({\bf x},\ \cdot \ )$ almost everywhere.  We conclude that for any ${\bf x}\in\bar M$, $\int_{M}|K_{P}({\bf x},{\bf y})|^{2}d{\bf y}<\infty$.  

Now consider $\phi({\bf x})=\int_{M}|K_{P}({\bf x},{\bf y})|^{2}d{\bf y}$.  The function $\phi$ is constant on orbits since the measure on $M$ is invariant;
\[\phi({\bf x}t)=\int_{M}|K_{P}({\bf x}t,{\bf y})|^{2}d{\bf y}=\int_{M}|K_{P}({\bf x},{\bf y}t^{-1})|^{2}d{\bf y}=\int_{M}|K_{P}({\bf x},{\bf y})|^{2}d{\bf y}=\phi({\bf x}).\]
\noindent
Thus $\phi$ descends to a function on $\bar M/G=\bar X$.  Since the map from $\bar M$ to $H^{-\infty}_c(\bar M)$ defined by ${\bf y}\mapsto \delta_{\bf y}$ is continuous, the composition
\[{\bf y} \mapsto P\delta_{\bf y}= K_P(\cdot, {\bf y})\]
is a continuous function $\bar M\to L^2(M)$.  We may conclude that $\phi:\bar X\to \mathbb R_{+}$ is continuous.  Denote by $\frac{d{\bf x}}{dt}$ the quotient measure on $X$.  The compactness of $\bar X$ together with continuity of $\phi$ imply that $\int_{X}\phi({\bf x})\frac{d{\bf x}}{dt}<\infty$. Thus we have that $K_{P}\in L^{2}(\frac{M\times M}{G})$.
 \end{proof}

Choosing a measurable global section $x$ in $M$ and representing points ${\bf x}\in M$, ${\bf x}\to (t,x)\in G\times X$, we obtain an isomorphism of measure spaces $(M,d{\bf x})\cong (G\times X, dt\otimes dx)$.  Whenever $P\in \mathcal B(L^{2}(M))^{G}$ and $K_{P}\in L^{2}_{\rm loc}(M\times M)$, this isomorphism and the criterion for invariance allow a representation
\[K_P({\bf x},{\bf y})\longrightarrow K_{P}(t,x;s,y) \stackrel {\rm def} =\kappa(ts^{-1};x,y),\quad s,t\in G, \ x,y \in X\]
\noindent
with $\kappa \in L^{2}_{\rm loc}(G\times X \times X).$
 
 \begin{lemma}\label{normkappa} Let $P\in \mathcal B(L^{2}(M))^{G}$.  Then ${\rm Tr}_{G}(P^{*}P)= \int_{\frac{M\times M}{G}}|K_{P}|^{2}$.\end{lemma}
 \begin{proof} Let $(\psi_{k})_{k}$ be an orthonormal basis for $L^{2}(X)$.  In the decomposition $L^{2}(M)\cong \bigoplus_{k}L^{2}(G)\otimes \psi_{k}$, the invariant operator $P$ has a matrix representation $P\to [L_{h_{kl}}]_{kl}$.  In terms of this, we compute ${\rm Tr}_{G}(P^{*}P)=\sum_{kl}\|h_{kl}\|_{L^{2}(G)}^{2}$.
 
Now, except on a set of measure zero, we have a description of $P$  
\[(Pu)({\bf x})=\int_{M}K_{P}({\bf x},{\bf y})u({\bf  y})d{\bf y}=(Pu)(t,x)=\int_{G\times X} ds dy\ \kappa(s;x,y)u(st,y).\]
The distributional kernels $h_{ij}$ can be recovered from $\kappa$ by projecting into the summands in $L^{2}(M)\cong \bigoplus_{l} (L^{2}(G)\otimes \psi_{l})$,
\[h_{ij} = \int_{X\times X} \ dxdy\ \kappa(\ \cdot \ ;x,y)\psi_{j}(y)\bar\psi_{i}(x).\]
Let us compute the norm of $\kappa$ in $L^{2}(G\times X\times X)$.  Since $(\psi_{j})_{j}$ is an orthonormal basis for $L^{2}(X)$, the set $(\bar \psi_{i}\otimes \psi_{j})_{ij}$ forms an orthonormal basis for $L^{2}(X\times X)$.  By construction, $h_{ij}$ is equal the $ij^{th}$ Fourier coefficient of $\kappa$ with respect to the decomposition $L^{2}(G\times X\times X)\cong \bigoplus_{ij}(L^{2}(G)\otimes \psi_{i}\otimes \psi_{j})$.  Hence $\sum_{ij} \|h_{ij}\|^{2}_{L^{2}(G)}=\|\kappa \|_{L^{2}(G\times X\times X)}^{2}$ and ${\rm Tr}_{G}(P^{*}P)=\|\kappa \|_{L^{2}(G\times X\times X)}^{2} = \int_{\frac{M\times M}{G}}|K_{P}({\bf x},{\bf y})|^{2}\ \frac{d{\bf x}d{\bf y}}{dt}$.\end{proof}
 
 \begin{corollary}\label{theworks} If $P\in \mathcal B(L^{2}(M))^{G}$ is an invariant self-adjoint projection such that ${\rm im}(P)\subset H^{\infty}(M)$, then ${\rm Tr}_{G}(P)<\infty$.\end{corollary}

\begin{rem}{\rm All the previous results extend trivially to operators acting in bundles.} \end{rem}

\begin{theorem}For $q>0$, the operator $\square$ on $M$ is $G$-Fredholm.\end{theorem}
\begin{proof} Let $\square =\int_{0}^{\infty}\lambda dE_{\lambda}$ be the spectral decomposition of $\square$ and for $\delta >0$, $P=\int_{0}^{\delta}dE_{\lambda}$.  Thus $\im(1-P)\subset\im\square$.  Further, $\im P\subset L^{2}(M,\Lambda^{p,q})$ is closed, invariant and, by Corollary \ref{bomb}, $\im P\subset H^{\infty}(M,\Lambda^{p,q})$.  Corollary \ref{theworks} implies that ${\rm codim}_{G}(\im(1-P))<\infty$.  The requirement on the kernel of $\square$ is verified noting that $\ker(\square) \subset\im P$ in the above.\end{proof}
\begin{corollary} If $q>0$, $\dim_G  L^{2}\bar H^{p,q}(M) < \infty$.                        
\end{corollary}
\begin{proof} By Lemma \ref{decomp} $L^{2}\bar H^{p,q}(M) = \ker(\square_{p,q}) = \im( E_{0})$ which has finite $G$-dimension.\end{proof}

\section{The $G$-Fredholm Property of $\square_b$}

In Theorem 5.4.9, \cite{FK} it is established that condition $Y(q)$ implies that
\[\|\phi\|_{H^{\epsilon-1}(bM)}^2\lesssim Q_b(\phi,\phi)\quad (\phi\in\mathcal B^{p,q})\]
\noindent
with $\epsilon=1/2$. It follows that for the boundary Laplacian, all the results from the preceding discussions go through without change. More generally, it is sufficient for subelliptic estimates of the form above to hold, with and $\epsilon>0$. In \cite{K2, K3}, it is shown that both for domains and CR manifolds, finite ideal $q$-type implies that subellipticity holds for $(p, q)$-forms.

\section{Applications}

\subsection{The $\dbar$- and $\dbar_b$-Neumann problems} The $G$-Fredholm property of $\square$ and $\square_b$ imply that we can sometimes solve the problems 
\begin{equation}\label{dbareqs}\dbar u = f \qquad \dbar_b u = f\end{equation}
\noindent
analogously to the compact case. Results from Sects. 2 and 3 of \cite{P2} give conditions guaranteeing solvability of these equations in some closed, $G$-invariant subspaces of $L^2(M)$. For example, if $L$ is such a subspace of $L^2(M)$ containing a form with compact support (even only in one component), then by Cor. 2.3 and Lemmata 3.5, 3.6 of \cite{P2}, we have $\dim_G L=\infty$ and so `most' elements $f\in L$ are good right-hand-sides for \eqref{dbareqs}. Similarly, if $L$ contains an essentially unbounded function which has bounded $L^2(G)$-norms on slices, then also $\dim_G L=\infty$, by Lemma 2.4 of the same article. 

\subsection{Levi Problem} In \cite{P2} we used the weaker results of \cite{P1} to give sufficient conditions for the solution of the Levi problem on strongly pseudoconvex $M$. It would be interesting to see if the results of \cite{BF} carry over to the present setting and whether the condition of amenability, introduced in \cite{P2} remains of importance here.

\section{Discrete cocompact group case}

\subsection{Criteria for finite $G$-dimensionality} In Lemmata \ref{kersmooth}, \ref{kerint} we deduced the finite $G$-dimensionality of closed, invariant subspaces $L\subset L^2(M)$ assuming that $L\subset H^s(M)$ for arbitrarily large $s$. In the event that the group in question has a discrete, cocompact subgroup, we need much less. We have not yet been able to find an application for this fact as even in the subelliptic case we have the strong hypotheses satisfied, as in Corollary \ref{applied}. Still, it seems of interest to ask whether the presence of such subgroups is important for the theorem to hold with the weaker hypotheses of $\epsilon$-gain or whether a proof exists that the theorem is true anyway. A possible characterization is in \cite{MS}, that of a central extension of a compact group. There the right- and left-invariant Sobolev spaces coincide.

\begin{example}\label{reals}{\rm Let us work out the analogue of Lemmata \ref{kersmooth} and \ref{kerint} completely and explicitly for an easy group. Let $G=\mathbb R$, and consider $L$, a closed, translation-invariant subspace of $L^{2}(\mathbb R)$ satisfying a condition $\|u\|_{\epsilon}\le C\|u\|_{0}$ for $u \in L$. As we have mentioned before, a projection $P: L^2(\mathbb R)\to L^2(\mathbb R)$ onto $L$ would take the form $u\mapsto Pu = h\ast u$ for some $h$, a distribution on $\mathbb R$. Taking the Fourier transform we find $\widehat{Pu} = \hat h\hat u$. Requiring $P=P^2$ implies that $h=h\ast h$, which translates to the condition that $\hat h$ satisfy $\hat h=\hat h^2$. In this case, $\hat h:\mathbb R\to\{0,1\}$, almost everywhere, so $h$ is completely characterized by $S={\rm supp}\ \hat h$. Let us now examine the estimate. By Plancherel's theorem, the condition $\|u\|_{\epsilon}\le C\|u\|_{0}$ for $u \in L$ is equivalent to $\|\Lambda^{\epsilon/2}\hat u\|_0\le C\|\hat u\|_{0}$ for $u \in L$ where $\Lambda(\xi)=(1+|\xi|^2)$. Concretely, 
\[\int_{\mathbb R}d\xi\ (1+|\xi|^2)^\epsilon|\hat u(\xi)|^2\le C^2\int_{\mathbb R}d\xi\ |\hat u(\xi)|^2\quad (u\in L)\]
\noindent
or, involving $\hat h=\chi_S$, we can consider all $u\in L^2(\mathbb R)$,
\[\int_{\mathbb R}d\xi\ [C^2-(1+|\xi|^2)^\epsilon]\ \chi_S(\xi)\ |\hat u(\xi)|^2\ge 0\quad (u\in L^2(\mathbb R)).\]
\noindent
Clearly this implies that $S=\{\xi\in\mathbb R\mid C^2-(1+|\xi|^2)^\epsilon\ge0\}$. Notice that 
\[S=\{\xi\in\mathbb R\mid C^{2/\epsilon}-1\ge |\xi|^2 \}=\{\xi\in\mathbb R\mid (C^{t})^{2/\epsilon t}-1\ge |\xi|^2\}.\]
\noindent
Therefore, all $u\in L$ also satisfy $\|u\|_{\epsilon t}\le C^{t}\|u\|_{0} \ \forall t\in \mathbb R$, with the constants' growth identical to that in Lemma \ref{iter}. We cannot invoke \cite{KN,N} this time to obtain analyticity here as they assume compactness, but instead apply the Paley-Wiener theorem.  This provides that $u\in L$ are analytic as their Fourier transforms have compact support, ${\rm supp}\ \hat u\subset{\rm supp}\ \hat h=S$ for all $u\in L$. We can also compute the $G$-dimension of $L$. We see that $L={\rm im}\ \check\chi_S\ast\cdot$, where $\check{}$ means the inverse Fourier transform. Applying Plancherel's theorem again, we obtain $ \|h\|^2=\|\check\chi_S\|^2 = meas(S)=2\sqrt{C^{2/\epsilon}-1}$.}\end{example}

The following Paley-Wiener-type is from \cite{GHS}.

\begin{lemma}\label{ghspw} Let $L$ be a closed $\Gamma$-invariant
subspace in $L^2(M,E)$, $L\subset W^\epsilon$ for some $\epsilon>0$ and there
exists $C>0$ such that 
\[\| u\|_\epsilon\leq C\| u\|_0,\quad (u\in L) \] 
\noindent
Then $\dim_\Gamma L<\infty$.\end{lemma}

To prove this lemma \cite{GHS} uses the following lemma about estimates of Sobolev norms on compact manifolds with boundary.

\begin{lemma} Let $X$ be a compact Riemannian manifold, possibly with a boundary. Let $E$
be a (complex) vector bundle with an hermitian metric over $\bar X$.
Denote by $\langle\cdot,\cdot\rangle$ the induced Hermitian inner product in the Hilbert
space $L^2(X,E)$ of square-integrable sections of $E$ over $X$. Denote by
$W^s=W^s(X,E)$ the corresponding Sobolev  space of sections of $E$ over $X$,
$\|\cdot\|_s$ the norm in this space. Let us choose
a complete orthonormal system $\{\psi_j;\ j=1,2,\dots\}$ in $L^2(X,E)$.
Then for all
$\epsilon>0$ and $\delta>0$ there exists an integer $N>0$ such that
\begin{equation}\label{epsgain}\| u\|_0\leq\delta\| u\|_\epsilon\ \hbox{provided}\ u\in W^\epsilon\
\hbox{and}\ \langle u,\psi_j\rangle=0,\;j=1,\dots,N.\end{equation}
\end{lemma} 
{\it Proof of Lemma \ref{ghspw}}. Let us choose a $\Gamma$-invariant covering
of $M$ by balls $\gamma B_k,\;k=1,\dots,m,\;\gamma\in\Gamma$, so that all the balls
have smooth boundary (e.g. have sufficiently small radii). Let us choose a
complete orthonormal system $\{\psi_j^{(k)}\;j=1,2,\dots\}$ in
$L^2(B_k,E)$ for every $k=1, \dots m$.  Then
$\{(\gamma^{-1})^\ast\psi_j^{(k)},\;j=1,2,\dots\}$ will be an orthonormal
system in $\gamma B_k$  (here we identify the element $\gamma$ with the corresponding
transformation of $M$).

   Given the subspace $L$ satisfying the conditions in the Lemma let us define a map
\[P_N:L\longrightarrow L^2\Gamma \otimes\mathbb C^{mN}\]
\[u\mapsto \{\langle u,(\gamma^{-1})^\ast\psi_j^{(k)}\rangle,\;j=1,2,\dots,N;\;k=1,\dots,m;\;
\gamma\in\Gamma\}\;.\]
Since $\dim_\Gamma L^2\Gamma\otimes\mathbb C^{mN}=mN<\infty$ the desired result will
follow if we prove that $P_N$ is injective for large $N$.
Assume that $u\in L$ and $P_Nu=0$.
Using Lemma 1.6 we get then
\[\| u\|_{0,\gamma B_k}^2\leq\delta_N^2\| u\|_{\epsilon,\gamma B_k}^2,\ \
k=1,\dots,m;\ \gamma\in\Gamma,\]
where $\delta_N\to 0$ as $N\to\infty$ and $\| \cdot\|_{s,\gamma B_k}$ means the
norm in the Sobolev space $W^s$ over the ball $\gamma B_k$.
  Summing over all $k$ and $\gamma$ we get
\[\| u\|_0^2\leq C_1^2\delta_N^2\| u\|_\epsilon^2\;,\]
where $C_1>0$ does not depend on $N$. This clearly contradicts (1.6) unless
$u=0$. $\square$

\noindent
Clearly this can be extended slightly using the results in \cite{Ad}.

\begin{rem} It is not necessary to require that $L$ is closed in $L^2$
in Proposition 1.5. For any $L$ satisfying (1.6) we can consider its
closure $\bar L$ in $L^2$. Then obviously $\bar L\subset W^\epsilon$ and Proposition 1.5
implies that $\dim_\Gamma\bar L<\infty$.

It is important to note that for the discrete group, the $\epsilon$-gain in \eqref{epsgain} is sufficient to obtain the finite $\Gamma$-dimensionality of $L$. Consider $\Gamma=\mathbb Z\subset\mathbb R=G$ and a function $h\in C^\infty_c(0,1)$. Clearly $\dim_\Gamma\langle h\rangle=1$ and an estimate like \eqref{epsgain} certainly holds. Taking the bigger group, the Paley-Wiener theorem implies that $\dim_G\langle h\rangle=\infty$, each of the generated spaces $\langle h\rangle$ according to the appropriate groups, {\it cf.} \cite{P2}. Sect. 2. This should be compared to the following fact. We have not been able to derive a similar property ($\epsilon$-gain implies finite-dimensionality) for continuous groups except for those that are central extensions of compact groups or reductive groups. Since all such groups have discrete cocompact subgroups \cite{M} we gain nothing new. \end{rem}

\subsection{Groups with Biinvariant Laplacians}

It is well-known that Abelian groups and compact semisimple groups all have Laplacians (second-order elliptic differential operators) that commute with left- and right-translations, \cite{M}.   
   
\begin{lemma}\label{iter} Let $G$ be a compact, semisimple Lie group, and suppose that $L$ is a right-invariant subspace of $L^2(G)$ satisfying
\[\|u\|_1\lesssim\|u\|_0 \quad (u\in L).\]
\noindent
It follows that $L$ is finite $G$-dimensional, finite-dimensional, and analytic.\end{lemma}
\begin{proof}  Since $G$ is compact, semisimple, it has a biinvariant elliptic Laplacian, $\Delta$. Since $L$ is right-invariant, $L$ can be written as the image of left-convolution by some distribution $h$ so that $L_h = L_h^* = L_h^2$. Abbreviate the convolution $L_h = h\ast\cdot$ and the invariant operator $\sqrt{1+\Delta } = \ ' $. 

Using the invariance of $\Delta$ and the identity $h=h\ast h$, it is easy to see that $h'= h' \ast h= h \ast h'$. Furthermore, $h'' = h' \ast h'$, and iterating, $h^{(k)} = (h')^{ \ast k}$. Now, using assumptions, identities, and the definition of the Sobolev norms, 
\[\|h'\ast h\ast u\|_0=\|h'\ast u\|_0=\|(h\ast u)'\|_0\stackrel {\rm def} = \|h\ast u\|_1\le C\|h\ast u\|_0 \quad (u\in L^2(G)).\]
\noindent
This means that the operator $L\ni u\mapsto h'\ast u$ has norm less than or equal to $C$. Similarly (and perhaps surprisingly), the previous line shows that $\|h'\ast u\|_0\le C\|h\ast u\|_0$ for all $u\in L^2(G)$. So $h'\ast u\in L^2(G)$, and this implies that $h'=h'\ast h$ acts on {\it it} in a bounded way:
\[\|h \ast u \|_{2} = \|h'' \ast u \|_{0} = \|h'\ast h'\ast u \|_{0} \le C \|h' \ast u \|_{0} \le C^2 \| h\ast u \|_{0} \quad (u\in L^2(G)).\]

Iterating this process shows that $L_{h} u \in H^{\infty}(G)$ and so is smooth for all $u\in L^2(G)$.  Therefore $h$ is smooth by Lemma \ref{kersmooth} and ${\rm dim}_G \ {\rm Im}\ L_{h}=\int_G |h|^2 = h(0) < \infty$. The finite-dimensionality of $L$ is just Rellich. 
  Noting the exponential growth of the constants: $\|L_{h}u\|_s \le C^s\|L_h u\|_0$ and the fact that the biinvariant Laplacian must have analytic coefficients (coming as is does from invariant vector fields) we conclude that $u \in {\rm im}\ L_{h}$ are in fact analytic, \cite{N}.\end{proof}
\begin{rem}{\rm With no change, the argument above works for a fractional ({\it i.e.} any $\epsilon>0$) gain $\|L_h u\|_\epsilon\le C\|L_h u\|_0$, but it fails if one drops the assumption that the Laplacian be biinvariant.}\end{rem}
\noindent 
The lemma has an obvious corollary. 
\begin{corollary}\label{weako} If $L$ satisfies the hypotheses of the previous lemma and $u\in L$ is nonzero, then $\{x\in M\mid u(x) = 0\}$ has no accumulation point. \end{corollary}
Corollary \ref{weako} is reminiscent of the Paley-Wiener theorem in that the small support of the Fourier transform $\hat u$ translates to analyticity (implying support almost everywhere) of the function $u$. See \cite{P2} for more about the Paley-Wiener theorem in this setting.
\begin{rem}Mimicking \cite{A, GHS} it is possible to make a correspondence between compactness of an operator on a compact space and the $\Gamma$-Fredholm property of the lift of that operator on a regular covering.\end{rem}

\section{Acknowledgments}

The author thanks M. Engli\v s, G. Folland, and E. Straube for helpful conversations.

\end{document}